\documentclass[12pt]{amsart}
\usepackage{amssymb,amsmath,amscd}
\usepackage{latexsym,bm,bbm,mathrsfs}
\usepackage{hyperref,graphicx}
\usepackage{mathtools}
\usepackage{stmaryrd}
\usepackage[all,pdf]{xy}
\usepackage{graphicx}
\graphicspath{ {./} }

\newtheorem{thm}{Theorem}[section]
\newtheorem{quest}[thm]{Question}

\newtheorem{theorem}[thm]{Theorem}
\newtheorem{cor}[thm]{Corollary}

\newtheorem{prop}[thm]{Proposition}
\newtheorem{lemma}[thm]{Lemma}

\newtheorem{remark}[thm]{Remark}

\title{Monodromy of rational curves on K3 surfaces of low genus}

\author{Sailun Zhan}
\address{Department of Mathematics, Indiana University, Bloomington, IN, 47405, U.S.A.}
\email{zhans@indiana.edu}

\begin{document}
\let\thefootnote\relax\footnote{\copyright 2022. This manuscript version is made available under the CC-BY-NC-ND 4.0 license https://creativecommons.org/licenses/by-nc-nd/4.0/. https://doi.org/10.1016/j.jpaa.2022.107021}
\begin{abstract}
In many situations, the monodromy group of enumerative problems will be the full symmetric group. In this paper, we study a similar phenomenon on the rational curves in $|\mathcal{O}(1)|$ on a generic K3 surface of fixed genus over $\mathbb{C}$ as the K3 surface varies. We prove that when the K3 surface has genus $g$, $1\leq g\leq 3$, the monodromy group is the full symmetric group.
\end{abstract}
\maketitle

\section{Introduction}

Monodromy problems are ubiquitous in geometry. In many situations, the monodromy group will be the full symmetric group. For example, for a generic plane curve of degree $d$, the only multiple tangent lines to the curve are either flexes or bitangent lines. By the Pl\"ucker formula, we know it has $3d(d-2)$ flexes and $\frac{1}{2}d^{4}-d^{3}-\frac{9}{2}d^{2}+9d$ bitangents. In \cite{Ha79}, Harris proved the following results. The monodromy group of the flexes of a generic degree $d$ plane curve is $ASL_{2}(\mathbb{Z}/3\mathbb{Z})$ if $d=3$, and is the full symmetric group if $d>3$. The monodromy group of the bitangents of a generic degree $d$ plane curve is $O_{6}(\mathbb{Z}/2\mathbb{Z})$ if $d=4$, and is the full symmetric group if $d>4$. The monodromy group of the lines on a generic degree $2n-3$ hypersurface in $\mathbb{P}^{n}$ is the odd orthogonal group $O_{6}^{-}(\mathbb{Z}/2\mathbb{Z})$ if $n=3$, and is the full symmetric group if $n>3$. The monodromy group of the conics tangent to five generic conics is the full symmetric group. For more references on this topic, see the recent survey paper by Frank Sottile and Thomas Yahl\cite{SY21}. For example, the Schubert problems on Grassmannians\cite{Vak06}\cite{RSSS06}, the enumerative problems from the Bernstein-Kushnirenko Theorem\cite{Est19}, the Fano problems\cite{HK20}, and the monodromy group of the projection of an irreducible nondegenerate curve\cite{PS05} or an irreducible and reduced hypersurface\cite{CCM20}.

Now consider the rational curves in the $g$-dimensional linear system $|\mathcal{O}(1)|$ on a K3 surface of genus $g$ over $\mathbb{C}$. By the Yau-Zaslow formula (\cite{YZ96} or \cite{Be99}), generically there will be 24 rational curves if $g=1$, 324 rational curves if $g=2$, and 3200 rational curves if $g=3$. When $g\geq 2$, this is related to universal Severi varieties. In \cite{CD12}, Ciliberto and Dedieu prove the irreducibility of universal Severi varieties parametrizing irreducible, reduced, nodal hyperplane sections of primitive K3 surfaces of genus $g$, with $3\leq g\leq 11$, $g\neq 10$. It is also conjectured that the universal family of all rational curves on K3 surfaces of a given genus is irreducible. See \cite{CD12} or \cite{De09}.

In this paper, we study the monodromy groups of the rational curves in the linear system $|\mathcal{O}(1)|$ on a generic K3 surface of genus $\leq 3$ over $\mathbb{C}$ as the K3 surface varies. 

The following definition of the monodromy group is from \cite[I.]{Ha79}. Let $X$ and $Y$ be irreducible algebraic varieties of the same dimension over $\mathbb{C}$, and $\pi:Y\to X$ a map of degree $d>0$. Let $p\in X$ be a generic point of $X$ and $\Gamma=\pi^{-1}(p)=\{q_1,...,q_d\}$ the fiber of $\pi$ over $p$. If we let $U\subset X$ be a suitable small Zariski open subset of $X$, $V=\pi^{-1}(U)\subset Y$, then the restricted map $\pi:V\to U$ will be an unbranched covering space. For any loop $\gamma:[0,1]\to U$ in $U$ with base point $p$ and any point $q_i\in\pi^{-1}(p)$, there is a unique lifting $\tilde{\gamma}_{i}$ of $\gamma$ to a loop in $V$ with $\tilde{\gamma}_{i}(0)=q_i$. We can define a permutation $\phi_{\gamma}$ of $\Gamma$ by sending each point $q_i$ to the endpoint of the lifted arc $\tilde{\gamma}_i$. Since $\phi_{\gamma}$ depends only on the homotopy class of $\gamma$, we have a homomorphism $\pi_{1}(U,p)\to S_d$. 

The image of this homomorphism is called the \emph{monodromy group} of the map $\pi:Y\to X$ since it does not depend on the choice of Zariski open $U$ as long as $\pi:V\to U$ is unramified. We prove the following Theorems.

\begin{theorem}\label{g1}
The monodromy group of the rational curves in the linear system $|\mathcal{O}(1)|$ on a generic K3 surface of genus 1 over $\mathbb{C}$ is the full symmetric group $S_{24}$.
\end{theorem}

For the genus 1 K3 surfaces, we are considering elliptic K3 surfaces with sections. These surfaces are determined by  the pairs $(A,B)\in H^{0}(\mathbb{P}^{1},\mathcal{O}_{\mathbb{P}^{1}}(8))\times H^{0}(\mathbb{P}^{1},\mathcal{O}_{\mathbb{P}^{1}}(12))$. If we denote $H^{0}(\mathbb{P}^{1},\mathcal{O}_{\mathbb{P}^{1}}(8))\times H^{0}(\mathbb{P}^{1},\mathcal{O}_{\mathbb{P}^{1}}(12))$ by $V$, and define
\[
Z=\{(A,B,t)\in V\times\mathbb{P}^{1}|4A^{3}(t)+27B^{2}(t)=0\},
\]
then the monodromy group of the rational curves is the monodromy group of the projection map $\pi\colon Z\to V$. See $\S 2$.

The next theorem is a direct corollary of Harris's work (\cite{Ha79}) on the monodromy group of the bitangents of a generic degree $d$ plane curves, since a generic primitive K3 surface of genus 2 is a double cover of $\mathbb{P}^{2}$ branching over a smooth sextic curve. See $\S 3$. 

\begin{theorem}\label{g2} 
The monodromy group of the rational curves in the linear system $|\mathcal{O}(1)|$ on a generic primitive K3 surface of genus 2 over $\mathbb{C}$ is the full symmetric group $S_{324}$.
\end{theorem}

We prove a similar statement for primitive K3 surfaces of genus 3.

\begin{theorem}\label{g3}
The monodromy group of the rational curves in the linear system $|\mathcal{O}(1)|$ on a generic primitive K3 surface of genus 3 over $\mathbb{C}$ is the full symmetric group $S_{3200}$.
\end{theorem}

To prove this theorem, we follow the idea in the proof of Theorem \ref{g2}. We first prove the monodromy group is doubly transitive. Then we determine the situation when we will have exactly one less rational curves. Finally, we show that the K3 surface of genus 3 satisfying such a situation does exist. The last two steps imply that the monodromy group contains a simple transposition, which completes the proof. See $\S 4.$

Note that in \cite{CD14}, some partial results on the monodromy group of rational curves for quartic K3 surfaces were obtained, and the following open question was raised.

\begin{quest}
Is the monodromy group of the rational curves in the linear system $|\mathcal{O}(1)|$ on a generic primitive K3 surface of genus $g\geq 2$ the full symmetric group?
\end{quest}

Our results answer this question in some low genus cases.

Using the relations between the Galois groups of the function fields and the monodromy groups, we also prove the following three corollaries which are of arithmetic interest.

Let $F$ be a number field.

\begin{cor}\label{c1}
There exists a Zariski-dense subset of $F$-rational points in $V=H^{0}(\mathbb{P}^{1},\mathcal{O}_{\mathbb{P}^{1}}(8))\times H^{0}(\mathbb{P}^{1},\mathcal{O}_{\mathbb{P}^{1}}(12))$ such that for every elliptic K3 surface over $F$ corresponding to a point in the subset, each of the $24$ nodal rational curves is defined over a degree $24$ non-Galois field extension of $F$.
\end{cor}

\begin{cor}\label{c2}
There exists a Zariski-dense subset of $F$-rational points in the linear space of sextic curves in $\mathbb{P}^2$ such that for every primitive K3 surface of genus 2 corresponding to a point in the subset, each of the $324$ nodal rational curves is defined over a degree $324$ non-Galois field extension of $F$.
\end{cor}

\begin{cor}\label{c3}
There exists a Zariski-dense subset of $F$-rational points in the linear space of quartic surfaces in $\mathbb{P}^3$ such that for every primitive K3 surface of genus 3 corresponding to a point in the subset, each of the $3200$ nodal rational curves is defined over a degree $3200$ non-Galois field extension of $F$.
\end{cor}

\begin{center}
\sc{Acknowledgements}
\end{center}
\vspace{0.1 in}
I thank my advisor Professor Michael Larsen for his guidance and valuable discussions throughout this work, and in particular, for suggesting the problem. I thank Professor Thomas Dedieu for pointing out a mistake in an earlier version of the paper. I also thank him for his helpful suggestions and comments. Finally, I thank the referee for the careful reading of the manuscript and for pointing out several inaccuracies.

\section{Elliptic K3 surfaces}

If a K3 surface $S$ has genus 1, then it has a line bundle $L$ such that $L^{2}=0$, which gives an elliptic fibration of $S$ over $\mathbb{P}^{1}$. In this paper, we will consider elliptic K3 surfaces with sections. From \cite{Hu16}(Chap. 11), \cite{Mi81} or \cite{Ka77}, we know that each minimal elliptic surface with a section over $\mathbb{P}^{1}$ corresponds to a unique Weierstrass fibration. Since $S$ is a K3 surface, it turns out that the corresponding Weierstrass fibration is a closed subscheme of the $\mathbb{P}^{2}$-bundle $\mathbb{P}(\mathcal{O}_{\mathbb{P}^{1}}(4)\oplus\mathcal{O}_{\mathbb{P}^{1}}(6)\oplus\mathcal{O}_{\mathbb{P}^{1}})$ over $\mathbb{P}^{1}$, which is defined by the Weierstrass equation:
\[
y^{2}z=x^{3}+Axz^{2}+Bz^{3}
\]
where $A\in H^{0}(\mathbb{P}^{1},\mathcal{O}_{\mathbb{P}^{1}}(8))$ and $B\in H^{0}(\mathbb{P}^{1},\mathcal{O}_{\mathbb{P}^{1}}(12))$. Moreover, $\Delta=4A^{3}+27B^{2}\in\mathcal{O}_{\mathbb{P}^{1}}(24)$ is not identically zero, and vanishes at $q\in\mathbb{P}^{1}$ if and only if the fiber is singular. For every $q\in\mathbb{P}^{1}$, either $v_{q}(A)\leq 3$ or $v_{q}(B)\leq 5$. Notice that this is equivalent to the statement that the Weierstrass fibration has only rational double points.

On the other hand, given $A\in H^{0}(\mathbb{P}^{1},\mathcal{O}_{\mathbb{P}^{1}}(8))$ and $B\in H^{0}(\mathbb{P}^{1},\mathcal{O}_{\mathbb{P}^{1}}(12))$ satisfying the above conditions, the Weierstrass equation defines a surface whose minimal desingularization is a K3 surface, which has the elliptic fibration with a section.

Denote the vector space $H^{0}(\mathbb{P}^{1},\mathcal{O}_{\mathbb{P}^{1}}(8))\times H^{0}(\mathbb{P}^{1},\mathcal{O}_{\mathbb{P}^{1}}(12))$ by $V$. Notice that the above conditions are open conditions in $V$. Hence for a generic choice of $A$ and $B$, the discriminant $\Delta$ vanishes at 24 distinct points of $\mathbb{P}^{1}$, which gives 24 nodal rational curves in $|\mathcal{O}(1)|$ on a K3 surface. If we define
\[
Z=\{(A,B,t)\in V\times\mathbb{P}^{1}|4A^{3}(t)+27B^{2}(t)=0\},
\]
then the map $\pi\colon Z\to V$ is generically finite, and our problem is to determine the monodromy group of this map.

Let us first consider the following construction, which can be found in \cite{Ke13}(\S 7).

Let $(a_{1},...,a_{12})\in\mathbb{C}^{12}$ be an ordered tuple of $12$ pairwise distinct points of $\mathbb{C}$, and $K$ be a nonzero complex number. Define $l:=-\sqrt[3]{\frac{27}{4}}\in\mathbb{R}$. Let $x_{0},x_{1}$ be homogeneous coordinates for $\mathbb{P}^{1}$ and let
\[
A=K^{2}lx_{1}^{8}\in H^{0}(\mathbb{P}^{1},\mathcal{O}_{\mathbb{P}^{1}}(8)),
\]
\[
B=\prod_{i=1}^{12}(x_{0}-a_{i}x_{1})+K^{3}x_{1}^{12}\in H^{0}(\mathbb{P}^{1},\mathcal{O}_{\mathbb{P}^{1}}(12)). 
\]
Then $A$ vanishes to order $8$ at $[1:0]$, whereas $B$ does not vanish at $[1:0]$, so the condition that either $\mu_{p}(A)\leq 3$ or $\mu_{p}(B)\leq 5$ is met. On the other hand, we have
\[
\Delta(A,B)=27\prod_{i=1}^{12}(x_{0}-a_{i}x_{1})\left(\prod_{i=1}^{12}(x_{0}-a_{i}x_{1})+2K^{3}x_{1}^{12}\right)\in H^{0}(\mathbb{P}^{1},\mathcal{O}_{\mathbb{P}^{1}}(24)),
\]
which is not identically $0$. For a generic choice of $a_{1},...a_{12}$ and $K$, $\Delta(A,B)$ vanishes at $24$ distinct points. Hence this $(A,B)$ defines a smooth K3 surface with $24$ nodal singular fibers.

\begin{lemma}\label{S12}
The monodromy group $M\subset S_{24}$ of the map $\pi: Z\to V$ contains a subgroup $M_{1}$ which is isomorphic to $S_{12}$. In addition, if we define the embedding $f\colon S_{12}\to S_{24}$ by
\[
f(\sigma)(i)=\begin{cases}
\sigma(i), & \text{if } i\leq 12;\\
\sigma(i-12)+12, & \text{if } i>12,
\end{cases}
\]
then $M_{1}$ is a conjugate of the image of $f$.
\end{lemma}

\begin{proof}
Fix a tuple $(a_{1},...,a_{12})\in\mathbb{C}^{12}$ and a $K\in\mathbb{C}$ in the above discussion, which gives a point $P\in V$. Let $Y\subset\mathbb{C}^{12}$ be the open subvariety consisting of pairwise dinstinct tuples $(c_{i})$ such that $(c_{i})$ and $K$ define a smooth K3 surface with $24$ nodal singular fibers. Fix an element $\sigma\in S_{12}$. Since $Y$ is irreducible, and hence connected, there is a path $\gamma(t)\in Y$ connecting $(a_{1},...,a_{12})$ and $(\sigma(a_{1}),...,\sigma(a_{12}))$. This path $\gamma(t)$ then gives a loop in $V$ at the point $P$. Let $\gamma(t)=(\gamma_{1}(t),...,\gamma_{12}(t))$. When $|K|$ is very small, for every value of $t$ and every $i$ the polynomial $\prod_{i=1}^{12}(x_{0}-\gamma_{i}(t)x_{1})+2K^{3}x_{1}^{12}$ will have exactly one simple root that is closer to $\gamma_{i}(t)$ than to $\gamma_{j}(t)$ for $j\neq i$. Therefore the monodromy group $M$ contains a subgroup $M_{1}$, which is a conjugate of Im$(f)$ by the expression of $\Delta(A,B)$. 
\end{proof}

Now we consider another similar construction as above. Let $(b_{1},...,b_{8})\in\mathbb{C}^{8}$ be an ordered tuple of $8$ pairwise distinct points of $\mathbb{C}$, and $K$ be a nonzero complex number. Define $l:=\sqrt{-\frac{4}{27}}\in\mathbb{C}$. Let $x_{0},x_{1}$ be homogeneous coordinates for $\mathbb{P}^{1}$ and let
\[
A=\prod_{i=1}^{8}(x_{0}-b_{i}x_{1})+K^{2}x_{1}^{8}\in H^{0}(\mathbb{P}^{1},\mathcal{O}_{\mathbb{P}^{1}}(8)). 
\]
\[
B=K^{3}lx_{1}^{12}\in H^{0}(\mathbb{P}^{1},\mathcal{O}_{\mathbb{P}^{1}}(12)),
\]
Then $B$ vanishes to order $12$ at $[1:0]$, whereas $A$ does not vanish at $[1:0]$, so the condition that either $\mu_{p}(A)\leq 3$ or $\mu_{p}(B)\leq 5$ is met. On the other hand, if we denote $\prod_{i=1}^{8}(x_{0}-b_{i}x_{1})$ by $\beta$, we have
\[
\Delta(A,B)=4\beta(\beta^{2}+3\beta K^{2}x_{1}^{8}+3K^{4}x_{1}^{16})\in H^{0}(\mathbb{P}^{1},\mathcal{O}_{\mathbb{P}^{1}}(24)),
\]
which is not identically $0$. For a generic choice of $b_{1},...b_{8}$ and $K$, $\Delta(A,B)$ vanishes at $24$ distinct points. Hence this $(A,B)$ defines a smooth K3 surface with $24$ nodal singular fibers.

\begin{lemma}
The monodromy group $M\subset S_{24}$ of the map $\pi\colon Z\to V$ contains a subgroup $M_{2}$ which is isomorphic to $S_{8}$. In addition, if we define the embedding $g:S_{8}\to S_{24}$ by
\[
g(\sigma)(i)=\begin{cases}
\sigma(i), & \text{if } i\leq 8;\\
\sigma(i-8)+8, & \text{if } 8<i\leq16;\\
\sigma(i-16)+16, & \text{if } i>16,
\end{cases}
\]
then $M_{2}$ is a conjugate of the image of $g$.
\end{lemma}

\begin{proof}
Fix a tuple $(b_{1},...,b_{8})\in\mathbb{C}^{8}$ and a $K\in\mathbb{C}$ in the above discussion, which gives a point $P\in V$. Let $Y\subset\mathbb{C}^{8}$ be the open subvariety consisting of pairwise dinstinct tuples $(c_{i})$ such that $(c_{i})$ and $K$ define a smooth K3 surface with $24$ nodal singular fibers. Fix an element $\sigma\in S_{8}$. Since $Y$ is irreducible, and hence connected, there is a path $\gamma(t)\in Y$ connecting $(b_{1},...,b_{8})$ and $(\sigma(b_{1}),...,\sigma(b_{8}))$. This path $\gamma(t)$ then gives a loop in $V$ at the point $P$. Let $\gamma(t)=(\gamma_{1}(t),...,\gamma_{8}(t))$. Notice that $\beta^{2}+3\beta K^{2}x_{1}^{8}+3K^{4}x_{1}^{16}=(\beta+\frac{3+\sqrt{-3}}{2}K^{2}x_{1}^{8})(\beta+\frac{3-\sqrt{-3}}{2}K^{2}x_{1}^{8})$. Let $\beta(t)=\prod_{i=1}^{8}(x_{0}-\gamma_{i}(t)x_{1})$. When $|K|$ is very small, for every value of $t$ and every $i$ the polynomials $\beta(t)+\frac{3+\sqrt{-3}}{2}K^{2}x_{1}^{8}$ and $\beta(t)+\frac{3-\sqrt{-3}}{2}K^{2}x_{1}^{8}$ will have exactly one simple root that is closer to $\gamma_{i}(t)$ than to $\gamma_{j}(t)$ for $j\neq i$. Therefore the monodromy group $M$ contains a subgroup $M_{2}$, which is a conjugate of Im$(g)$ by the expression of $\Delta(A,B)$. 
\end{proof}

A permutation group $G$ acting on a non-empty finite set $X$ is called \emph{primitive} if $G$ acts transitively on $X$ and $G$ preserves no nontrivial partition of $X$.

\begin{lemma}
The monodromy group $M\subset S_{24}$ of the map $\pi\colon Z\to V$ is primitive.
\end{lemma}

\begin{proof}
To prove the transitivity, we can assume $M$ actually contains $f(S_{12})$ and a conjugate of $g(S_{8})$. Then we have two $S_{12}$ orbits $A_{1}=\{1,2,...,12\}$, $A_{2}=\{13,14,...,24\}$ and three $S_{8}$ orbits $B_{1},B_{2},B_{3}$. For any two elements $a,b\in\{1,...,24\}$, we want to prove that there exists an element $\sigma\in M$ such that $\sigma(a)=b$. If $a$ and $b$ both lie in $A_{1}$ or both lie in $A_{2}$, then such an element exists. Now suppose $a$ lies in $A_{1}$ and $b$ lies in $A_{2}$. Since each of the orbits $B_{1},B_{2},B_{3}$ has cardinality $8$, there must be an orbit, say $B_{1}$, which contains an element $c\in A_{1}$ and an element $d\in A_{2}$. Then there are elements in $M$ which send $a$ to $c$, $c$ to $d$, and $d$ to $b$.

Giving a partition is the same as giving an equivalence relation, and each equivalence class is a piece of the partition. Suppose now we have an equivalence relation on $\{1,2,...,24\}$. If $A_{1}$ contains two elements $a,b$ in the same class and an element $c$ in a different class, then there is an element in $M$ which interchanges $b,c$ and keeps $a$ unchanged. We deduce that $M$ does not preserve the relation. A similar argument works for $A_{2}$. On the other hand, we know that if a transitive action preserves a partition, then all parts have the same size. Hence without loss of generality, we only have the following two possible equivalence relations which may be preserved by $M$.
\begin{gather*}
\{1,2,...,12\}\{13,14,..,24\},\\
\{1,13\},\{2,14\},...,\{12,24\}.
\end{gather*}
For the first possiblity, since there is an orbit, say $B_{1}$, which contains an element $a\in \{1,2,...,12\}$ and an element $b\in \{13,14,...,24\}$, there is an element $\sigma\in M_{2}$ which interchange $a$ and $b$. Notice that we can always choose a $\sigma$ such that there is another element $c\in \{1,2,...,24\}$ which is fixed by $\sigma$. Hence $M$ does not preserve the relation.

Now we consider the second possibility. Each of the three $S_{8}$ orbits $B_{1},B_{2},B_{3}$ cannot contain a whole class. Otherwise, it will contain two elements in the same class and another element in the different class. Hence $B_{1}$ consists of $8$ elements coming from $8$ different classes. Suppose the relation is preserved by $M$. Then the rest of the $8$ elements in those $8$ classes must form an $S_{8}$ orbit, say $B_{2}$. So the remaining $4$ classes form the last orbit $B_{3}$, which is impossible.

Hence $M$ preserves no nontrivial partition of $\{1,2,...,24\}$ and the action is transitive, which implies that $M$ is primitive. 
\end{proof}

Since we have the classification of the finite group acting primitively on $24$ elements, we are able to determine the group $M$.

\noindent{\bf Proof of Theorem \ref{g1}.} By \cite[Table 1]{BL96}, we know that there are only $5$ groups that act primitively on $24$ elements: $PSL_{2}(23), PGL_{2}(23), M_{24}, A_{24}$ and $S_{24}$. Since $M$ contains $S_{12}$, $|M|>12!$. Hence $M$ can only be $A_{24}$ or $S_{24}$. But since $M$ contains a conjugate of $g(S_{8})$, it will contain some elements not in $A_{24}$. Hence $M=S_{24}$. $\square$

\noindent{\bf Proof of Corollary \ref{c1}.} Consider the map $\pi\colon Z\to V$ such that both the varieties and the map are defined over $F$. We have the following map.
\[
K(V)\to K(Z)\to K(\tilde{Z}),
\]
where $K(V),K(Z)$ are the function fields of $V$ and $Z$, and $K(\tilde{Z})$ is the Galois closure of $K(Z)/K(V)$ in an algebraic closure, which can be realized as the function field of a Galois cover $\tilde{Z}$ of the variety $V$.
  
By \cite[$\S$I.]{Ha79} and $V$ being normal, we know that the monodromy group $M$ is a subgroup of the Galois group $Gal(K(\tilde{Z})/K(V))\subset S_{24}$. But we have $M=S_{24}$ by Theorem \ref{g1}. Hence we deduce that $Gal(K(\tilde{Z})/K(V))=S_{24}$, and $K(Z)$ is a degree $24$ non-Galois field extension of $K(V)$. By the Hilbert irreducibility theorem, for the $F$-valued point $x$ in a Zariski-dense subset of $V(F)$, the residue field extension $F\cong K(x)\to K(Z_{x})$ is also a degree $24$ non-Galois field extension. Hence for every elliptic K3 surface over $F$ in this subset, each of the $24$ nodal rational curves is defined over a degree $24$ non-Galois field extension of $F$. $\square$

\section{Double covers of the projective plane}

If a K3 surface $S$ has genus 2, then it has a line bundle $L$ such that $L^{2}=2$, which gives a degree 2 morphism from $S$ to $\mathbb{P}^{2}$ branching over a sextic curve \cite[Chap.2, Remark 2.4]{Hu16}. For a generic $S$, it will branch over a smooth sextic curve $C$, and the 324 nodal rational curves in $|L|=|\mathcal{O}(1)|$ are exactly the preimages of the 324 bitangents of the curve $C$ by the Riemann-Hurwitz formula.  Now we let $W$ be the linear system of plane curves of degree 6, $\Delta\subset\mathbb{P}^{2}\times\mathbb{P}^{2}$ the diagonal and $J_{0}=(\mathbb{P}^{2}\times\mathbb{P}^{2})\backslash\Delta$, and let $J\subset W\times J_{0}$ be the incidence correspondence
\[
J=\{(C, (p_{1},p_{2}))|m_{p_{i}}(C\cdot\overline{p_{1}p_{2}})\geq 2\},
\]
where $m_{p}$ is the intersection number at $p$. We want to study the projection morphism $\pi\colon J\to W$ and ask for the monodromy group $M$ of the bitangents. This is solved by Harris in the following theorem.

\begin{theorem}\label{Harris}
\cite[II.5.]{Ha79} The monodromy group $M$ on the bitangents of a generic curve $C$ of degree $d\geq 5$ is the full symmetric group.
\end{theorem}

\noindent{\bf Proof of Theorem \ref{g2}.} By the above discussion and Theorem \ref{Harris}, we deduce that the monodromy group $M$ is the full symmetric group $S_{324}$. $\square$

\noindent{\bf Proof of Corollary \ref{c2}.} The proof is similar to that of Corollary \ref{c1} using the Hilbert irreducibility theorem. $\square$

Before we move on to the next section, let us briefly review the idea in the proof of Theorem \ref{Harris} since we will be using the same strategy when we prove the genus 3 case. The argument has two steps \cite[II.3.]{Ha79}. The first step is to prove $M$ is doubly transitive, and the second step is to prove $M$ contains a simple transposition. Then it will follow that $M$ contains all the simple transpositions and $M$ is the full symmetric group.

The proof of the first step uses the irreducibility of certain correspondence spaces, which is straightforward. However, the proof of the second step relies on the following lemma.

\begin{lemma}\label{lemma}
\cite[II.3.]{Ha79} Let $\pi\colon Y\to X$ be a holomorphic map of degree $n$. If there exists a point $p\in X$ such that the fiber of $Y$ over $p$ consists of exactly $n-1$ distinct points (i.e. $n-2$ simple points $q_{1},\ldots,q_{n-2}$ and one double point $q_{n-1}$) and if $Y$ is locally irreducible at $q_{n-1}$, then the monodromy group $M$ of $\pi$ contains a simple transposition.
\end{lemma}

By the Pl\"ucker formulas, a smooth plane curve $C$ will have exactly one fewer bitangent than the generic case if it possesses a simple flex bitangent, which is a bitangent $\overline{p_{1}p_{2}}$ with $m_{p_{1}}(C\cdot\overline{p_{1}p_{2}})=2$ and $m_{p_{2}}(C\cdot\overline{p_{1}p_{2}})=3$ meeting transversely away from $p_{1}$ and $p_{2}$, and no other lines $l$ with $m_{C}(l)\geq 3$, where
\[
m_{C}(l):=\sum_{q}(m_{q}(C\cdot l)-1)
\]
and the sum is taken over all the tangent points of $l$ with $C$.

Hence the last part of the argument is to prove the existence of such a curve, namely, a smooth sextic curve with 322 regular bitangents and 1 simple flex bitangent. In other words, the simple flex bitangent appears when exactly two bitangents coincide. This is done by dimension arguments regarding certain correspondence spaces.

\section{Smooth quartic surfaces in $\mathbb{P}^{3}$}

If a K3 surface has genus 3, then it has a line bundle $L$ such that $L^{2}=4$, which gives a degree 1 morphism from the surface to a quartic surface in $\mathbb{P}^{3}$. Now if we fix a quartic surface $S$ in $\mathbb{P}^{3}$, then the curves in $|\mathcal{O}(1)|$ are exactly the intersections of the planes in $\mathbb{P}^{3}$ with $S$. For a generic smooth quartic surface, the 3200 nodal rational curves in $|\mathcal{O}(1)|$ are the intersections of $S$ with those planes which are tangent to $S$ at three points. Now we let $W$ be the linear system of quartic surfaces in $\mathbb{P}^{3}$, $D\subset\mathbb{P}^{3}\times\mathbb{P}^{3}\times\mathbb{P}^{3}$ be the closed subset that consists of the triples in the fat diagonal and those collinear points. Let $J_{0}=(\mathbb{P}^{3}\times\mathbb{P}^{3}\times\mathbb{P}^{3})\backslash D$, and let $J\subset W\times J_{0}$ be the incidence correspondence
\[
J=\{(S, (p_{1},p_{2},p_{3}))|S \text{ is singular at }p_{i} \text{ or }T_{p_{i}}S=\overline{p_{1}p_{2}p_{3}}\text{ for }i=1,2,3\},
\]
where $T_{p_{i}}S$ is the projective tangent plane of $S$ at $p_{i}$, and $\overline{p_{1}p_{2}p_{3}}$ is the projective plane passing through the points $p_{1}, p_{2}, p_{3}$. We want to study the projection morphism $\pi\colon J\to W$ and ask for the monodromy group $M$ of those projective planes.

As is explained in the preceding section, we first need to prove $M$ is doubly transitive.

\begin{lemma}\label{doubly}
The monodromy group $M$ is doubly transitive.
\end{lemma}

\begin{proof}
Notice that since the fibers of $\eta:J\to J_{0}$ are irreducible linear spaces of codimension 9 in $W$ and $J_{0}$ is irreducible, $J$ is irreducible. So the inverse image $\pi^{-1}(U)\subset J$ of any Zariski open subset $U\subset W$ is connected. Hence $M$ is transitive. 

Now we set
\[
W'=\pi(\eta^{-1}(p_{1},p_{2},p_{3}))=\{S|S \text{ is singular at }p_{i} \text{ or }T_{p_{i}}S=\overline{p_{1}p_{2}p_{3}}\text{ for }i=1,2,3\}\subset W,
\]
\[
J'=\{(S,q_{1},q_{2},q_{3})\in W'\times J_{0}|\ S \text{ is singular at }q_{i} \text{ or }T_{q_{i}}S=\overline{q_{1}q_{2}q_{3}}\text{ for }i=1,2,3,
\]
\[
\ q_{i}\neq p_{j},\ \overline{q_{1}q_{2}q_{3}}\neq\overline{p_{1}p_{2}p_{3}}\},\text{ where }J'\subset W'\times J_{0}. 
\]

Since the fibers of $J'$ over $J_{0}$ are irreducible linear subspaces of codimension 9 in $W'$ and $J_{0}$ is irreducible, $J'$ is irreducible. Now since $J'$ contains all the projective planes in the fiber of $\pi$ over $S$ except $\overline{p_{1}p_{2}p_{3}}$, it follows that given two pairs of tangent planes $(\overline{p_{1}p_{2}p_{3}},\overline{q_{1}q_{2}q_{3}})$ and $(\overline{p_{1}p_{2}p_{3}},\overline{r_{1}r_{2}r_{3}})$, there is an element $\sigma\in M$ sending $\overline{p_{1}p_{2}p_{3}}$ to $\overline{p_{1}p_{2}p_{3}}$ and sending $\overline{q_{1}q_{2}q_{3}}$ to $\overline{r_{1}r_{2}r_{3}}$. Combining this with the transitivity of $M$, we deduce that $M$ is doubly transitive.
\end{proof}

Next, we need to find a simple transposition in $M$. Recall that in the genus 2 case, we used the Pl\"ucker formulas to determine when the number of bitangents will be exactly one less than that of the generic case. For the genus 3 case, there are two ways to determine when the number of rational curves will be exactly one less than 3200. The first method uses compactified jacobians. The second method, which uses the geometry of smooth quartic surfaces in more detail, will be presented in the next section.

We define \emph{tangent curves} as the intersections of the surface with its projective tangent planes. Now we want to ask what will happen if two tangent curves with three nodes coincide. It turns out that we will have a tangent curve with two nodes and one cusp.

\begin{theorem}\label{trinodes}
Suppose there exists a smooth quartic surface $S$ which has a tangent curve $C$ with two nodes and one cusp, and all the other rational tangent curves have three nodes. Suppose every hyperplane section of $S$ is irreducible. Then there are 3198 tangent curves with three nodes.
\end{theorem}

\begin{proof}
Suppose we know there exists such a smooth quartic surface. Then by \cite[Prop. 3.3 \& Prop. 4.5]{Be99}, the Euler characteristic of the compactified jacobian of a rational curve with three nodes is 1, and the Euler characteristic of the compactified jacobian of a rational curve with two nodes and one cusp is 2. Hence by \cite[Cor. 2.3]{Be99}, the number of rational curves with three nodes is $3200-2=3198$.
\end{proof}

\begin{prop}\label{transposition}
Under the same assumption as in Theorem \ref{trinodes}, the monodromy group $M$ contains a simple transposition.
\end{prop}

\begin{proof}
Let $W$ be the variety of quartic surfaces in $\mathbb{P}^3$, and let 
\[
\tilde{J}=\{(S, \overline{(p_{1},p_{2},p_{3})})| S \text{ is singular at }p_{i} \text{ or }T_{p_{i}}S=\overline{p_{1}p_{2}p_{3}}\text{ for }i=1,2,3\}\subset W\times (J_{0}/S_3). 
\]
Then there is a point $p\in W$ such that the fiber of the projection $\pi\colon \tilde{J}\to W$ over $p$ consists of exactly $3199$ distinct points (3198 curves with three nodes and 1 curve with two nodes and one cusp) by Theorem \ref{trinodes}. On the other hand, since the fibers of $\eta\colon J\to J_0$ are linear spaces and $J_0$ is smooth, $J$ is also smooth and hence locally irreducible at every point. In particular, $\tilde{J}=J/S_3$ is locally irreducible at the curve with two nodes and one cusp. Hence by Lemma \ref{lemma}, the monodromy group $M$ contains a simple transposition. 
\end{proof}

The last thing we need to prove is that such a smooth quartic surface does exist. For that purpose, we analyze a certain space of the quartic surfaces in $\mathbb{P}^{3}$ which have a tangent curve with two nodes and one cusp. Let $g(x,y,z)$ be a quartic homogeneous polymial, which represents an irreducible quartic curve with two nodes and one cusp in $\mathbb{P}^{2}$. Such a curve does exist. See \cite[The proof of Thm. 4.5.4(ii) and p. 400]{GLS18}. 

\begin{lemma}\label{irreducible}
Consider the family $X$ of quartic surfaces $g(x,y,z)+wf(x,y,z,w)=0$, where $f(x,y,z,w)$ is a cubic homogeneous polynomial. For a generic member in this family, every hyperplane section is irreducible.
\end{lemma}

\begin{proof}
The family $X$ is parametrized by $\mathbb{A}^{20}$. Consider the intersections of a fixed hyperplane $H$ and the quartic surfaces in the family. The resulting family of quartic curves $Y$ is parametrized by $\mathbb{A}^{10}$. If the curve is reducible, then it is either a product of a line and a cubic or a product of two quadrics. We call the first case type I and the second case type II. 

Now consider quartic homogeneous polynomials of type I, which are parametrized by $\mathbb{P}^{2}\times\mathbb{P}^{9}$. Without loss of generality, suppose $H$ has the form $ax+by+cz+dw=0$, where $a,b,c,d\in\mathbb{C}$ and $a\neq 0$. Eliminating $x$ from the equation of the family shows that the cone $G(y,z,w)+wF(y,z,w)$ passes through $H\cap X$, where $G$ is a fixed quartic and $F$ is any cubic. Let $L$ be a line and $K$ be a cubic in the variables $y,z,w$. If $LK=G+wF$ for some $F$. Then $w|(LK-G)$, or $\bar{L}\bar{K}=\bar{G}$ in the ring $\mathbb{C}[y,z]$. Hence $L$ lies in a 1-dimensional family and $K$ lies in a 6-dimensional family, which implies that the intersection of $\mathbb{P}^{2}\times\mathbb{P}^{9}$ with $\mathbb{A}^{10}$ in the ambient space $\mathbb{P}^{14}$ has dimension $7$. If there is a line lying simultaneously on the quartic surface and a hyperplane in $\mathbb{P}^{3}$, then there will be a 1-dimensional family of hyperplanes such that the line lies simultaneously on the quartic surface and every hyperplane in this family. Hence the subvariety of $\mathbb{A}^{20}$ parametrizing the quartic surfaces in $X$ which admit hyperplane sections of type I is at most $7+10+3-1=19$ dimensional. 

Now consider quartic homogeneous polynomials of type II, which are parametrized by $\mathbb{P}^{5}\times\mathbb{P}^{5}$. By a similar argument as above, the intersection of $\mathbb{P}^{5}\times\mathbb{P}^{5}$ with $\mathbb{A}^{10}$ in the ambient space $\mathbb{P}^{14}$ has dimension $6$. Hence the subvariety of $\mathbb{A}^{20}$ parametrizing the quartic surfaces in $X$ which admit hyperplane sections of type II is at most $6+10+3=19$ dimensional. 

Hence for a generic member in $X$, every hyperplane section is irreducible. 
\end{proof}

We are going to use the following facts \cite[Thm. 2.48 \& Cor. 2.49]{GLS07}.

Let $f\in m^{2}:=(x,y)^2\subset \mathbb{C}[[x,y]]$ and $k\geq 1$. Let $\mu(f)$ be the Milnor number of $f$. Define $crk(f):=2-rank (H(f)(0))$, where $H(f)(0)$ is the Hessian matrix of $f$ at $(0,0)$. Then
\[
\begin{aligned}
&(1)\ f \text{ is of type $A_{1}$ iff}\ crk(f)=0,\\
&(2)\ f \text{ is of type $A_{k}$ iff}\ crk(f)=1 \text{ and}\ \mu(f)=k,\ k\geq 2, \\
&(3)\ crk(f) \text{ and}\ \mu(f)\ \text{are semicontinuous on}\ m^2.
\end{aligned}
\]

\begin{lemma}\label{generic}
For a generic quartic surface in $X$, except for the tangent curve cut by $w=0$, all the tangent curves are of the following types:

(1) a curve with three nodes;

(2) a curve with two nodes or a curve with a node and a cusp;

(3) a curve with one node or a curve with one cusp or a curve with one tacnode.
\end{lemma}

\begin{proof}
We prove using case by case analysis.

Case (1). Let
\[
I_{3}=\{(p_{1},p_{2},p_{3})\in \mathbb{P}^{3}\times \mathbb{P}^{3}\times \mathbb{P}^{3}\ |\ p_{1}, p_{2}, p_{3}\notin g(x,y,z),\text{ distinct and not collinear}\}.
\]

Define $J_{3}\subset X\times I_{3}$ by 
\[
J_{3}=\{(S,p_{1},p_{2},p_{3})\ |\ S \text{ is singular at }p_{i} \text{ or }T_{p_{i}}S=\overline{p_{1}p_{2}p_{3}}\text{ for }i=1,2,3\}
\]
The fibers of $\eta\colon J_{3}\to I_{3}$ are codimensional 9 linear spaces in $X$, which can be checked by reducing to the case when $p_1=(0:0:0:1),p_2=(1:0:0:1)$, and $p_3=(0:1:0:1)$. Hence $J_3$ is irreducible of dimension equals to $\dim X$. The subvariety $J'_{3}\subset J_3$ defined by 
\[
J'_{3}=\{(S,p_{1},p_{2},p_{3})\ |\ S \text{ is singular at }p_{i} \text{ or }T_{p_{i}}S=\overline{p_{1}p_{2}p_{3}}\text{ for }i=1,2,3,
\]
\[
\text{and }p_{1} \text{ is at best a cuspidal singularity for the tangent curve}\}
\]
is thus either equals to $J_3$ or of dimension $< \dim J_3$. Since we can construct a surface in $X$ which has nodal singularities at three tangent points, the latter must hold. Hence the image of the projection of $J'_3$ in $X$ has dimension $<\dim X$. Replacing $p_1$ in the definition of $J'_3$ by $p_2$ and $p_3$, we know that for a generic quartic surface in $X$, the tangent curve with three singularities will have three nodes.

Case (2). Set 
\[
I_{2}=\{(p_{1},p_{2}, H)\in \mathbb{P}^{3}\times \mathbb{P}^{3}\times\check{\mathbb{P}^{3}}\ |\ p_{1}, p_{2}\notin g(x,y,z),\text{ distinct points, and }p_{1},p_{2}\in H\}.
\]
Define $J_{2}\subset X\times I_{2}$ by 
\[
J_{2}=\{(S,p_{1},p_{2}, H)\ |\ S \text{ is singular at }p_{i} \text{ or }T_{p_{i}}S=H\text{ for }i=1,2,3\}.
\]
As before, $J_2$ is irreducible of dimension equals to $\dim X+1$, and the subvariety $J'_2\subset J_2$ defined by
\[
J'_2=\{(S,p_{1},p_{2})\ |\ S \text{ is singular at }p_{i} \text{ or }T_{p_{i}}S=H\text{ for }i=1,2,3,
\]
\[
\text{and }p_{1} \text{ is at best a cuspidal singularity for the tangent curve}\}
\]
is of dimension at most $\dim X$. One can check that the fibers of $J'_2\to I_2$ are irreducible quadrics, which implies that $J'_2$ is irreducible. Hence the subvariety $J''_2\subset J'_2$ defined by
\[
J''_2=\{(S,p_{1},p_{2})\ |\ S \text{ is singular at }p_{i} \text{ or }T_{p_{i}}S=H\text{ for }i=1,2,3,\text{ and }p_{1} \text{ is at best an } A_3 
\]
\[
\text{ singularity for the tangent curve or both points have at best cuspidal singularities}\}
\]
is of dimension $<\dim X$. Replacing $p_1$ in the definitions of $J'_2$ and $J''_2$ by $p_2$, we know that for a generic quartic surface in $X$, the tangent curve with two singularities will have two nodes or one node and one cusp. 

Case (3). Set $I_{1}=\{(p_1, H)\in\mathbb{P}^{3}\times\check{\mathbb{P}^3}|\ p_{1}\notin g(x,y,z),\text{ and }p_{1}\in H\}$. Define $J_{1}\subset X\times I_{1}$ by 
\[
J_{1}=\{(S,p_{1},H)\ |\ p_{1}\in S\ \text{and $p_{1}$ is at best an $A_{4}$ singularity for the tangent curve}\}.
\]
By a similar argument as in case (2), one can check that $J_1$ is of dimension $<\dim X$, which implies that for a generic quartic surface in $X$, the tangent curve with one singularity will have at worst a tacnode ($A_{3}$) singularity. 
\end{proof}

\begin{prop}\label{existence}
There exists a smooth quartic surface $S$ which has a tangent curve $C$ with two nodes and one cusp, all the other rational tangent curves have three nodes, and every hyperplane section of $S$ is irreducible.
\end{prop}

\begin{proof}
Such a surface is generic in the family $X$ by Lemma \ref{irreducible} and Lemma \ref{generic}. 
\end{proof}

\noindent{\bf Proof of Theorem \ref{g3}.} By Lemma \ref{doubly}, the monodromy group $M$ is doubly transitive. Then by Theorem \ref{trinodes}, Proposition \ref{transposition} and Proposition \ref{existence}, the monodromy group $M$ contains a simpe transposition. Hence $M$ contains all the simple transpositions and is the full symmetric group. $\square$

\noindent{\bf Proof of Corollary \ref{c3}.} The proof is similar to that of Corollary \ref{c1} using the Hilbert irreducibility theorem. $\square$

\section{Geometry of smooth quartics}

In this section, we present a more geometric way to prove Theorem \ref{trinodes}.

\begin{prop}\label{integral}
For a generic quartic surface in $\mathbb{P}^{3}$, every hyperplane section is irreducible and reduced.
\end{prop}

\begin{proof}
The idea is similar to that of Lemma \ref{irreducible}. Quartic surfaces in $\mathbb{P}^{3}$ are parametrized by $\mathbb{P}^{34}$. If we intersect a generic quartic surface with a fixed hyperplane $H$, then we will get a quartic curve in $\mathbb{P}^{2}$. But quartic curves in $\mathbb{P}^{2}$ are parametrized by $\mathbb{P}^{14}$. Hence we have the following fiber bundle.
\[
\begin{CD}
\mathbb{P}^{34}\backslash Z @<<< \mathbb{A}^{20}\\
 @VV\pi V\\
\mathbb{P}^{14}
\end{CD}
\]
where $Z$ parametrizes the quartics which contain $H$, and $\pi$ maps the quartic to its intersection with $H$.

On the other hand, if the resulting quartic curve is reducible, then it is either a product of a line and a cubic or a product of two quadrics. We call the first case type I and the second case type II. 

A product of a line and a cubic in $\mathbb{P}^{2}$ is parametrized by $\mathbb{P}^{2}\times\mathbb{P}^{9}$, which is 11-dimensional. The fiber of $\pi$ is 20-dimensional, and this implies the variety parametrizing the quartics whose intersection with $H$ is of type I has dimension 31. Now fix a quartic surface. If there is a line lying simultaneously on the quartic surface and a hyperplane in $\mathbb{P}^{3}$, then there will be a 1-dimensional family of hyperplanes such that the line lies simultaneously on the quartic surface and every hyperplane in this family. That is to say, the intersection of the quartic surface and every hyperplane in the family is of type I. Combining this with the fact that the hyperplanes in $\mathbb{P}^{3}$ are parametrized by $\mathbb{P}^{3}$, we deduce that the variety parametrizing the quartic surfaces which admit hyperplane sections of type I is at most $31+3-1=33$ dimensional.

A product of two quadrics in $\mathbb{P}^{2}$ is parametrized by $\mathbb{P}^{5}\times\mathbb{P}^{5}$, which is 10-dimensional. The fiber of $\pi$ is 20-dimensional, and this implies the variety parametrizing the quartics whose intersection with $H$ is of type II has dimension 30. Since the hyperplanes in $\mathbb{P}^{3}$ are parametrized by $\mathbb{P}^{3}$, we deduce that the variety parametrizing the quartic surfaces which admit hyperplane sections of type II is at most $30+3=33$ dimensional.

But quartic surfaces in $\mathbb{P}^{3}$ are parametrized by $\mathbb{P}^{34}$. Hence for a generic quartic surface in $\mathbb{P}^{3}$, every hyperplane section is irreducible and reduced. 
\end{proof}

In particular, all the tangent curves are irreducible and reduced for a generic quartic surface. On the other hand, only certain types of singularities can occur for those tangent curves. Notice that all the tangent curves have arithmetic genus 3.

\begin{prop}\label{type}
\cite[Prop. 2.1.7]{Wi14} Let $p$ be a point on a generic smooth quartic hypersurface $S\subset\mathbb{P}^{3}$ and let $E_{p}\subset S$ be the tangent curve at $p$. Then $E_{p}$ is of one of the following type.
\begin{itemize}
\item (generic point)
 
$E_{p}$ has genus 2 and has one node at $p$
\item (simple parabolic point) 

$E_{p}$ has genus 2 and has one cusp at $p$ 
\item (simple Gauss double point)

$E_{p}$ has genus 1 and has two nodes, and $p$ is one of the nodes
\item (parabolic Gauss double point)

$E_{p}$ has genus 1 and has one node and one cusp, and $p$ is its cusp
\item (dual to parabolic Gauss double point)

$E_{p}$ has genus 1 and has one node and one cusp, and $p$ is its node 
\item (Gauss swallowtail)

$E_{p}$ has genus 1 and has one tacnode at $p$
\item (Gauss triple point)

$E_{p}$ has genus 0 and has three nodes, and $p$ is one of the nodes
\end{itemize}
\end{prop}

\begin{remark}
In \cite{Wi14}, this proposition is stated for a very general smooth quartic surface, since the condition of $Pic(S)\cong\mathbb{Z}$ is assumed there to make sure that all the hyperplane sections are irreducible and reduced. But by Proposition \ref{integral}, we know this is a generic condition.
\end{remark}

The references for the following assertions are \cite{Wi14} and \cite{Pi78}. When the quotation mark is used, it is referred to the classification in Proposition \ref{type}.

The generic points on $S$ are ``generic points". Let $\phi\colon S\to{\mathbb{P}^{3}}^{\vee}$ be the Gauss map which sends the point on the surface to its projective tangent plane in the dual space. Then the degree of the dual surface $S^{*}:=\phi(S)$ is 36. The ramification locus of $\phi$ is a smooth curve of degree 32 which consists of ``simple parabolic points", ``parabolic Gauss double points" and ``Gauss swallowtails". This curve is called the parabolic curve or the cuspidal curve. If we denote this curve by $C_{par}$, then the generic points on $C_{par}$ are ``simple parabolic points". There are 1920 ``parabolic Gauss double points" and 320 ``Gauss Swallowtails" on $C_{par}$. If we denote the dual curve by $C_{par}^{*}:=\phi(C_{par})$, then $\phi|_{C_{par}}\colon C_{par}\to C_{par}^{*}$ ramifies at ``Gauss swallowtails" and $C_{par}^{*}$ is smooth outside of the images of ``Gauss swallowtails". The curve $C_{par}^{*}$ has degree 96.

There is another important curve on $S$, which is called the double-cover curve. It is an irreducible curve which consists of ``simple Gauss double points", ``parabolic Gauss double points", ``dual to parabolic Gauss double points", ``Gauss swallowtails" and ``Gauss triple points". We denote this curve by $C_{d}$. The generic points on $C_{d}$ are ``simple Gauss double points". There are 9600 ``Gauss triple points", 1920 ``parabolic Gauss double points", 1920 ``dual to parabolic Gauss double points" and 320 ``Gauss swallowtails" on $C_{d}$. Notice that the 9600 ``Gauss triple points" give us the $9600/3=3200$ nodal rational curves in $|\mathcal{O}(1)|$. $C_{d}$ has nodal singularities at ``Gauss triple points", cuspidal singularities at ``dual to parabolic Gauss double points" and is smooth elsewhere. If we denote the dual curve by $C_{d}^{*}:=\phi(C_{d})$, then $\phi|_{C_{d}}\colon C_{d}\to C_{d}^{*}$ ramifies at ``parabolic double points" and ``Gauss swallowtails". $C_{d}^{*}$ is smooth outside of the images of ``Gauss triple points" and ``parabolic double points". If we call the images of ``Gauss triple points" triple points and the images of ``parabolic double points" cuspidal double points, then $C_{d}^{*}$ has three-branched nodes at triple points and cusps at cuspidal double points. $C_{d}^{*}$ has degree 480.

The parabolic curve $C_{par}$ intersects with the double-cover curve $C_{d}$ at ``parabolic Gauss double points" and ``Gauss swallowtails". The intersections at ``parabolic Gauss double points" are transversal. But $C_{par}$ is tangent with $C_{d}$ at ``Gauss swallowtails" with multiplicity 2. For more details, see \cite[$\S 2.5.3$]{Wi14}.

Now what we are interested in are the formulas which relate those numerical terms. There are plenty of them (see \cite{Pi78} and \cite[Chap. IV \& V]{Kl77}). But we only need one of them in the following form.

\begin{prop}\label{formula}
\cite[$\S$5 Theorem 4 (III$_{b}$)]{Pi78} Let $S$ be a smooth quartic surface. Using the notations as above, we have
\[
b(\mu_{0}-2)=\rho+\delta_{b}+v_{b}''+i, \tag{$*$}
\]
where $b$ is the degree of $C_{d}^{*}$, $\mu_{0}$ is the degree of $S^{*}$, $\rho$ is the degree of $C_{d}$, $\delta_{b}=g_{a}(C_{d})-g(C_{d})$ is the difference between the arithmetic genus and the geometric genus of $C_{d}$, $v_{b}''$ is the degree of the ramification divisor of the normalization map $\tilde{C_{d}^{*}}\to C_{d}^{*}$, and $i=C_{d}\cdot C_{par}$.
\end{prop}

Let us check this formula in the generic case. We have $b=480$ and $\mu_{0}$=36, so the left-hand side is
\[
b(\mu_{0}-2)=480\times(36-2)=16320.
\]
For the right-hand side, we have $\rho=320$. Since $C_{d}$ has 9600 nodal singularities and 1920 cuspidal singularities, $\delta_{b}=9600+1920=11520$. Since $C_{d}^{*}$ has 1920 cuspidal singularities, $v_{b}''=1920$. Finally, since $C_{par}$ and $C_{d}$ intersect transversally at 1920 ``parabolic Gauss double points" and intersect at 320 ``Gauss swallowtails" with multiplicity 2, we have $i=1920+2\times 320=2560$. Hence the right-hand side is
\[
\rho+\delta_{b}+v_{b}''+i=320+11520+1920+2560=16320,
\]
which is equal to the left-hand side.

Now suppose there exists a smooth quartic surface $S$ which admits a tangent curve with two nodes and one cusp, and all the other tangent curves are of the types in Proposition \ref{type}.  We call such a surface a \emph{good} surface. In this case, we will still have
\[
\mu_{0}=36,
\]
since the degree of the dual surface $S^{*}$ of a hypersurface $S$ of degree $d$ in $\mathbb{P}^{3}$ is $d(d-1)^{2}$ by \cite[P.362, IV,70]{Kl77}. On the other hand, we will also have
\[
b=480.
\]
The argument is as follows. By \cite[Prop. 1.4.4]{Wi14}, we have
\[
[C_{d}]+2[C_{par}]=\phi^{*}[S^{*}]+\phi^{*}[K_{\mathbb{P}^{3}}],
\]
which is an equality in the Chow group of 1-cycles on S, i.e. $CH_{1}(S)$. We know that $C_{d}\in\mathcal{O}_{S}(80)$ and $C_{par}\in\mathcal{O}_{S}(8)$ by \cite[Application 3, Lemma 7]{Pi78}. Since $\phi_{*}C_{d}=2C_{d}^{*}$, by the projection formula we get
\[
\begin{aligned}
(C_{d}+2C_{par})\cdot C_{d}&=(\phi^{*}S^{*}+\phi^{*}K_{\mathbb{P}^{3}})\cdot C_{d}\\
(C_{d}+2C_{par})\cdot C_{d}&=2(S^{*}\cdot C_{d}^{*})+2(K_{\mathbb{P}^{3}}\cdot C_{d}^{*})\\
80\times 80\times 4+2\times 8\times 80\times 4&=2\times 36\times {\rm deg}(C_{d}^{*})-2\times 4\times {\rm deg}(C_{d}^{*})\\
30720&=64\ {\rm deg}(C_{d}^{*})\\
{\rm deg}(C_{d}^{*})&=480.
\end{aligned}
\]

Hence in this case the left-hand side of the equation $(*)$ does not change, which is still 16320.

Now let us look at the right-hand side of the equation $(*)$ in this case. The degree $\rho=320$ of $C_{d}$ and the intersection number $i=80\times 8\times 4=2560$ stay the same. But we will see that $v_{b}''$ and $\delta_{b}$ are different.

We first claim that there are still 320 ``Gauss swallowtails" in this case. Since all the singularity types of the tangent curves in our surface are still the ones in Proposition 4.3, for each $p\in S$ there exist analytic coordinate charts $\varphi\colon (\mathbb{C}^{2},0)\to(S,p)$ and $\psi\colon (\mathbb{C}^{3},0)\to(\check{\mathbb{P}^{3}},x)$ such that the germ at 0 of $\psi^{-1}\circ\phi\circ\varphi$ is one of the following (\cite[Prop. 2.1.1]{Wi14} or \cite[(3.4)]{MS84}).
\[
\begin{aligned}
(nonparabolic\ germ)&\bullet (x,y)\mapsto(x,y,0),\\
(general\ parabolic\ germ)&\bullet (x,y)\mapsto(x^{3},x^{2},y),\\
(swallowtail\ germ)&\bullet (x,y)\mapsto(3x^{4}+x^{2}y,2x^{3}+xy,y),
\end{aligned}
\]
which is equivalent to the fact that the Gauss map has the following two properties (\cite[273--274]{MS84}).
\[
\begin{aligned}
&{\rm(i)\ zero\ is\ a\ regular\ value\ of}\ K,\\
&{\rm(ii)\ the\ cusps\ of\ the\ Gauss\ map\ are\ nondegenerate, }
\end{aligned}
\]
where $K$ is the determinant function of certain matrix associated with the projective second fundamental form \cite[p. 261]{MS84}.

But this implies that our smooth quartic surface $S$ has 320 ``Gauss swallowtails" by the following theorem, since ``Gauss swallowtails" are precisely the cusps of the Gauss map (\cite[$\S$2.2]{Wi14}).

\begin{theorem}
\cite[(2.5)]{MS84} Let $S\subset\mathbb{P}^{3}$ be a complex algebraic surface of degree $d$ satisfying the above two properties. Then there are precisely $2d(d-2)(11d-24)$ cusps of the Gauss map.
\end{theorem}

\begin{remark}
The above conditions are also equivalent to the statement that the projection map $\Gamma:=\{(x,H)\in S\times(\mathbb{P}^{3})^{\vee}|x\in H\}\to(\mathbb{P}^{3})^{\vee}$ is locally stable. See \cite[\S 3]{MS84} or \cite[\S 2.1]{Wi14}.
\end{remark}

To summarize what we have done so far, we have

\begin{lemma}\label{summary}
Let $S$ be a good surface. Then $S$ has 320 ``Gauss swallowtails". Using the notation in Proposition \ref{formula}, we have
\[
b=480;\ \mu_{0}=36;\ \rho=320;\ i=2560.
\]
\end{lemma}

Now in order to determine $\delta_{b}$ and $v_{b}''$, we need to look more carefully at the Gauss map $\phi$. 

\begin{lemma}\label{node}
Let $S$ be a good surface. Let $p,q\in C$ be the nodal singularity points. Then the germ of the singularity of $C_{d}$ at $p$ or $q$ is
\[
\{(x,y)\in\mathbb{C}^{2}|y^{3}=x^{2}\}\cup \{(x,y)\in\mathbb{C}^{2}|y=-x\}
\]
\end{lemma}

\begin{proof}
To determine the singularity of $C_{d}$ at $p,q$, we recall the Gauss map near the points $p,q$ and $r$, where $r\in C$ is the cusp. Note that the three branches in the dual space are in general position.
\[
\begin{aligned}
(near\ p)&\bullet\phi_{p}: (x,y)\mapsto(x,y,0),\\
(near\ q)&\bullet\phi_{q}: (x,y)\mapsto(x,y,x+y),\\
(near\ r)&\bullet\phi_{r}: (x,y)\mapsto(x^3,x^2,y).
\end{aligned}
\]
Considering the intersection of the local images of $\phi_{p}$ and $\phi_{q}$, we get a curve
\[
(x,-x,0),\ x\in\mathbb{C}.
\]
Its preimage under $\phi_{p}$ is
\[
(x,-x),\ x\in\mathbb{C}.
\]
Considering the intersection of the local images of $\phi_{p}$ and $\phi_{r}$, we get a curve
\[
(x^3,x^2,0),\ x\in\mathbb{C}.
\]
Its preimage under $\phi_{p}$ is
\[
(x^3,x^2),\ x\in\mathbb{C}.
\]
Hence the germ of the singularity of $C_{d}$ at $p$ has the form stated in the lemma. The germ of the singularity of $C_{d}$ at $q$ can be calculated similarly. 
\end{proof}

\begin{lemma}\label{cusp}
Let $S$ be a good surface. Let $r\in C$ be the cusp singularity point. Then $C_{d}$ has a tacnode at $r$.
\end{lemma}

\begin{proof}
By the calculations in Lemma \ref{node}, the intersection of the local images of $\phi_{p}$ and $\phi_{r}$ is
\[
(x^3,x^2,0),\ x\in\mathbb{C}.
\]
Its preimage under $\phi_{r}$ is
\[
(x,0),\ x\in\mathbb{C}.
\]
The intersection of the local images of $\phi_{q}$ and $\phi_{r}$ is
\[
(x^3,x^2,x^3+x^2),\ x\in\mathbb{C}.
\]
Its preimage under $\phi_{r}$ is
\[
(x,x^3+x^2),\ x\in\mathbb{C}.
\]
Hence the singularity of $C_{d}$ at $r$ is a tacnode. 
\end{proof}

\begin{lemma}\label{image}
Let $S$ be a good surface. Let $p\in C$ be a nodal singularity point. Then the singularity of $C_{d}^{*}$ at $\phi(p)$ has three local branches:
\begin{align*}
&\{(x,y,z)\in\mathbb{C}^{3}|y^{3}=x^{2}, z=0\},\\
&\{(x,y,z)\in\mathbb{C}^{3}|y^{3}=x^{2}, z=x+y\},\\
&\{(x,y,z)\in\mathbb{C}^{3}|y=-x, z=0\}.
\end{align*}
\end{lemma}

\begin{proof}
This follows from Lemma \ref{node} and Lemma \ref{cusp}. 
\end{proof}

\begin{lemma}\label{ramification}
(1) If a curve $C$ has a trinodal singularity at $p$, then the normalization map $\pi\colon\tilde{C}\to C $ does not ramify at the preimages of $p$.  

(2) If a curve $C$ has a cusp singularity at $p$, then the ramification divisor of the normalization map $\pi\colon\tilde{C}\to C$ has degree 1 at $\pi^{-1}(p)$.

(3) If a curve $C$ has a singularity at $p$ which is analytically isomorphic to the singularity in Lemma \ref{image}, then the ramification divisor of the normalization map $\pi\colon\tilde{C}\to C$ has degree 2.
\end{lemma}

\begin{proof}
(1) Let $q\in\tilde{C}$ be a point over the trinodal singularity $p\in C$. Consider the following exact sequence.
\[
(\pi^{*}\Omega_{C})_{q}\to\Omega_{\tilde{C},q}\to(\Omega_{\tilde{C}/C})_{q}\to 0.
\]
After we restrict the map $f\colon (\pi^{*}\Omega_{C})_{q}\to\Omega_{\tilde{C},q}$ to the fibers at $q$, we have the map $n/n^{2}\to m/m^{2}$, where $n$ is the maximal ideal in $\mathcal{O}_{C,p}$ and $m$ is the maximal ideal in $\mathcal{O}_{\tilde{C},q}$. This map is surjective as $\mathbb{C}$-linear spaces since the branches of the singularity are smooth. Hence $f$ is surjective by Nakayama's Lemma, which implies that $(\Omega_{\tilde{C}/C})_{q}=0$. Hence the normalization map $\pi\colon\tilde{C}\to C $ does not ramify at the preimages of $p$.  

(2) Since this is a local problem, we consider the normalization map
\[
\pi\colon Spec\ \mathbb{C}[t]\to Spec\ \mathbb{C}[x,y]/(y^{2}-x^{3}),
\]
\[
x\mapsto t^{2},\ \ y\mapsto t^{3}.
\]
Denote $\mathbb{C}[t]$ by $A$ and $\mathbb{C}[x,y]/(y^{2}-x^{3})$ by $B$. Then we have
\[
\begin{aligned}
\Omega_{\tilde{C}/C}&=\Omega_{A/B}\\
&=\Omega_{(B[t]/(t^{2}-x,\ t^{3}-y))/B}\\
&=(\mathbb{C}[t]/(2t,\ 3t^{2}))dt.
\end{aligned}
\]
Since our maximal ideal m in this case is $(t)$, we have length $(\Omega_{\tilde{C}/C})_{m}=1$.

(3) Let $q\in\tilde{C}$ be the point over $p$ whose neighborhood is mapped to the smooth branch of the singularity $p$. Then the normalization map $\pi\colon\tilde{C}\to C $ does not ramify at $q$ by the same argument in the proof of (1).

Now we look at the other points $r,s\in\tilde{C}$ over $p$ whose neighborhood is mapped to a cuspidal curve. By the calculation in (2), each of the points $r,s$ contribute 1 to the degree of the ramification divisor. Hence the ramification divisor of the normalization map $\pi\colon\tilde{C}\to C$ has degree 2. 
\end{proof}

\begin{lemma}\label{parabolic}
Let $S$ be a good surface. Then there are 1918 tangent curves with a node and a cusp.
\end{lemma}

\begin{proof}
Recall that $C_{par}$ and $C_{d}$ intersect transversally at ``parabolic Gauss double points" and intersect at ``Gauss swallowtails" with multiplicity 2 (\cite[$\S 2.5.3$]{Wi14}). Note that at the point $r\in C_{d}$, where $r\in C$ is the cuspidal point, $C_{par}$ intersect $C_{d}$ at $r$ with multiplicity 2 since $C_{d}$ has a tacnode at $r$ by Lemma \ref{cusp}. Now we calculate the number of ``parablic Gauss double points".
\[
\begin{aligned}
C_{d}\cdot C_{par}&=\#\{{\rm parabolic\ Gauss\ double\ points}\}+2\#\{{\rm Gauss\ swallowtails}\}+2\\
2560&=\#\{{\rm parabolic\ Gauss\ double\ points}\}+640+2\\
1918&=\#\{{\rm parabolic\ Gauss\ double\ points}\}.
\end{aligned}
\]
Hence we have 1918 ``parablic Gauss double points" in this case. 
\end{proof}

\noindent{\bf Proof of Theorem \ref{trinodes}.} We first calculate $v_{b}''$, which is the degree of the ramification divisor of the normalization map $\tilde{C_{d}^{*}}\to C_{d}^{*}$. $C_{d}^{*}$ has 1918 cuspidal singularities since there are 1918 tangent curves with a node and a cusp by Lemma \ref{parabolic}. The singularities are cusps by a similar argument as in Lemma \ref{node}. $C_{d}^{*}$ also has 1 singularity of the type in Lemma \ref{image}. All the other singularities are trinodal singularities coming from the tangent curves with three nodes. Hence by Lemma \ref{ramification}, we have
\[
v_{b}''=1918+2=1920.
\]

Now all the numerical terms in the following formula from Proposition \ref{formula} are known except for $\delta_{b}$ by Lemma \ref{summary} and the above argument.
\[
b(\mu_{0}-2)=\rho+\delta_{b}+v_{b}''+i.
\]
Hence we can calculate $\delta_{b}$.
\[
\begin{aligned}
16320&=320+\delta_{b}+1920+2560\\
\delta_{b}&=11520.
\end{aligned}
\]

Finally we look at $\delta_{b}$, which is the difference between the arithmetic genus and the geometric genus of $C_{d}$. Suppose there are $t$ tangent curves with three nodes. Then they will contribute $3t$ to $\delta_{b}$. There are 1918 tangent curves with a node and a cusp, which will contribute 1918 to $\delta_{b}$, since there are cuspidal singularities at ``dual to parabolic Gauss double points". The last three singularities come from the tangent curve with two nodes and one cusp, and we use Lemma \ref{node} and Lemma \ref{cusp}. At the cusp, $C_{d}$ has a tacnodal singularity, which contributes 2 to $\delta_{b}$. At each of the two nodes, $C_{d}$ has a singularity that is analytically isomorphic to $\{y^{3}-x^{2}\}\cup\{y=-x\}$, which contributes 3 to $\delta_{b}$. Hence we have
\[
\begin{aligned}
3t+1918+2+3\times 2&=11520\\
t&=3198.
\end{aligned}
\]  
$\square$

\end{document}